\documentclass[a4paper,12pt]{article}
\usepackage{amsmath}
\usepackage{amsthm}
\usepackage{amssymb}
\usepackage{amscd}
\usepackage{graphicx}
\usepackage{epsfig}
\usepackage[matrix,arrow,curve]{xy}
\usepackage{color}
\usepackage{mathrsfs}
\usepackage[scr=rsfs,cal=boondox]{mathalpha}

\usepackage{bbm}
\usepackage{stmaryrd}
\usepackage{hyperref}
\usepackage{multirow}

\usepackage{latexsym}
\usepackage{amsfonts}
\input xy
\usepackage{tikz}
\usetikzlibrary{matrix}
\usetikzlibrary{decorations.pathreplacing,angles,quotes}
\newtheorem{theorem}{Theorem}[section]
\newtheorem{proposition}[theorem]{Proposition}
\newtheorem{lemma}[theorem]{Lemma}

\theoremstyle{definition}
\newtheorem{example}[theorem]{Example}
\newtheorem{remark}[theorem]{Remark}
\newtheorem{definition}[theorem]{Definition}

\mathchardef\za="710B  %\alpha
\mathchardef\zb="710C  %\beta
\mathchardef\zg="710D  %\gamma
\mathchardef\zd="710E  %\delta
\mathchardef\zve="710F %\epsilon
\mathchardef\zz="7110  %\zeta
\mathchardef\zh="7111  %\eta
\mathchardef\zvy="7112 %\theta
\mathchardef\zi="7113  %\iota
\mathchardef\zk="7114  %\kappa
\mathchardef\zl="7115  %\lambda
\mathchardef\zm="7116  %\mu
\mathchardef\zn="7117  %\nu
\mathchardef\zx="7118  %\xi
\mathchardef\zp="7119  %\pi
\mathchardef\zr="711A  %\rho
\mathchardef\zs="711B  %\sigma
\mathchardef\zt="711C  %\tau
\mathchardef\zu="711D  %\upsilon
\mathchardef\zvf="711E %\phi
\mathchardef\zq="711F  %\chi
\mathchardef\zc="7120  %\psi
\mathchardef\zw="7121  %\omega
\mathchardef\ze="7122  %\varepsilon
\mathchardef\zy="7123  %\vartheta
\mathchardef\zf="7124  %\varomega
\mathchardef\zvr="7125 %\varrho
\mathchardef\zvs="7126 %\varsigma
\mathchardef\zf="7127  %\varphi
\mathchardef\zG="7000  %\Gamma
\mathchardef\zD="7001  %\Delta
\mathchardef\zY="7002  %\Theta
\mathchardef\zL="7003  %\Lambda
\mathchardef\zX="7004  %\Xi
\mathchardef\zP="7005  %\Pi
\mathchardef\zS="7006  %\Sigma
\mathchardef\zU="7007  %\Upsilon
\mathchardef\zF="7008  %\Phi
\mathchardef\zW="700A  %\Omega

\newcommand{\be}{\begin{equation}}
\newcommand{\ee}{\end{equation}}

\newcommand{\bea}{\begin{eqnarray}}
\newcommand{\eea}{\end{eqnarray}}
\newcommand{\beas}{\begin{eqnarray*}}
\newcommand{\eeas}{\end{eqnarray*}}
\def\*{{\textstyle *}}
\newcommand{\T}{{\mathbb T}}

\newcommand{\SU}{SU(2)}

\newcommand{\we}{\wedge}

\newcommand{\ot}{\otimes}

\newcommand{\La}{\big\langle}
\newcommand{\Ra}{\big\rangle}

\newcommand{\N}{\mathbb{N}}
\newcommand{\Z}{\mathbb{Z}}
\newcommand{\R}{\mathbb{R}}

\newcommand{\C}{\mathbb{C}}

\newcommand{\pa}{\partial}
\newcommand{\ti}{\times}

\newcommand{\cG}{{\mathcal G}}

\newcommand{\Ll}{{\pounds}}

\def\ran{\rangle}

\def\tU{{\widetilde{U}}}
\def\tM{\widetilde{M}}

\def\tR{{{\cR_\tU}}}
\def\th{{\zh_\tU}}

\def\bt{{\boxtimes}}

\def\op{\oplus}

\def\cL{{\mathcal L}}

\def\cR{{\mathcal R}}

\def\cO{{\mathcal O}}

\def\cU{{\mathcal U}}

\def\wh{\widehat}

\def\ol{\overline}
\def\ul{\underline}
\def\Sec{\operatorname{Sec}}

\def\la{\langle}
\def\ran{\rangle}
\def\bh{{\mathbf h}}

\def\sD{{\mathsf D}}

\def\sT{{\mathsf T}}

\def\sU{{\mathsf U}}

\def\sh{\mathsf{h}}

\def\xd{\mathrm{d}}
\def\xi{\mathrm{i}}

\def\cF{{\mathcal F}}

\newdir{|>}{%
!/4.5pt/@{|}*:(1,-.2)@^{>}*:(1,+.2)@_{>}}

\def\Graph{\operatorname{graph}}

\def\L{\mathbb{L}}

\def\SU{\operatorname{SU}}

\def\P{\mathbf{P}}

\newdir{ (}{{}*!/-5pt/@^{(}}
\newcommand{\id}{\mathrm{id}}

\def\N{\mathbb{N}}

\def\n{\nabla}

\newcommand{\m}{{\medskip}}
\newcommand{\mn}{{\medskip\noindent}}

\newcommand{\no}{{\noindent}}
\def\Rt{{\R^\ti}}

\newcommand{\lobar}[1]{\bar{\smash{#1}\vphantom{x}}\vphantom{#1}}
\newcommand{\ubar}[1]{\ul{\smash{#1}\vphantom{x}}\vphantom{#1}}
\newcommand{\bti}{\lobar{\times}}
\newcommand{\uti}{\ubar{\times}}

\def\Graph{\operatorname{graph}}
\begin{document}
\title{The regularity and products in contact geometry\footnote{This research was partially funded by the National Science Centre (Poland) within the project WEAVE-UNISONO, No. 2023/05/Y/ST1/00043.}}
%{Regular complete contact manifolds and their products}%:\\ a generalization of the Boothby-Wang theorem}
\author{Katarzyna Grabowska\footnote{email:konieczn@fuw.edu.pl }\\
\textit{Faculty of Physics,
                University of Warsaw}
\\ \\
Janusz Grabowski\footnote{email: jagrab@impan.pl} \\
\textit{Institute of Mathematics, Polish Academy of Sciences}}
\date{}
\maketitle
\begin{abstract} We study regular contact manifolds $(M,\zh)$ whose Reeb vector field is complete and prove that they are canonically principal bundles with the structure group $S^1$ or $\R$. For compact $M$, our proof is very short and elementary and covers the celebrated Boothby-Wang theorem, but we do not assume compactness from the very beginning. However, to prove our result in full generality we use some topological tools adapted to smooth fibrations. In the second part of the paper, we describe a natural concept of contact products of general contact manifolds as well as a product of principal contact manifolds, which exists if the periods of the Reeb vector fields are commensurate and corresponds to the construction of products of prequantization bundles of symplectic manifolds.

\par


\end{abstract}
\section{Introduction}
The main object of our studies in is the structure of regular contact manifolds. More precisely, let $\zh$ be a contact form on a manifold $M$, which will be assumed to be connected throughout this paper. We say that the contact manifold $(M,\zh)$ is \emph{regular} if the foliation $\cF_\cR$ of $M$ by orbits of the Reeb vector field $\cR$ is simple, i.e., the space $N=M/\cF_\cR$ of orbits (denoted also with $M/\cR$) has a manifold structure such that the canonical projection $p:M\to N$ is a smooth fibration. Here, by orbits of $\cR$ (which is a nonvanishing vector field on $M$) we understand the 1-dimensional submanifolds of $M$, being the images of trajectories of $\cR$.

\m The study on the structure of compact regular contact manifolds $(M,\zh)$ was originated by Boothby and Wang \cite{Boothby:1958} and then continued by many authors, \cite{Boyer:2008,Boyer:2007,Geiges:2008,Kegel:2021,Niederkruger:2005,Wadsley:1975} to list a few of them. The original result in \cite{Boothby:1958} says that on a compact regular contact manifold, there exists an equivalent contact form $\zh'$ whose Reeb vector field $\cR'$ is periodic, thus inducing a principal action of the group $S^1=\sU(1)$ on $M$. However, the proof in \cite{Boothby:1958} is incomplete in one important respect. This fact was noticed and corrected (cf. \cite{Geiges:2008,Niederkruger:2005}). Moreover, a relation to Riemannian geometry was discovered in \cite{Sullivan:1978,Wadsley:1975} as well as a generalization for Besse contact manifolds, which led to an orbifold structure on the space $M/\cR$.

In the regular case, the orbits of the Reeb vector field are closed submanifolds, so for compact $M$ they all are circles and $\cR$ is periodic on each orbit, however, \emph{a priori} with different periods.
We show, without the compactness assumption, that all minimal periods are the same and $\cR$ is automatically the fundamental vector field of a free $\sU(1)$-action, so we do not need to seek for a rescaling of $\cR$ as it is done in \cite{Boothby:1958}. The Boothby-Wang theorem for non-compact contact manifolds is, in principle, known; for instance, it was derived \cite{Kegel:2021} from results \cite{Wadsley:1975} on Riemannian metrics on $M$ such that the orbits of $\cR$ are geodesics. However, the power of our proof is that it is short and elementary, while a related proof in \cite{Wadsley:1975} takes several pages.

What is more, our proof generalizes to the situation when the Reeb vector field is complete, which is automatically satisfied in the compact case. It is clear that \emph{a priori} $\cR$ may have both, compact and non-compact orbits, as for instance the Reeb vector field on a punctured sphere $S^{2n+1}$. We show, however, that for complete $\cR$ we have only two possibilities: either $\cR$ has no periodic orbits, or all orbits are closed and share the same minimal period. To this end we use some topological results concerning smooth fibrations \cite{Meigniez:2002}. Our main result is the following.
\begin{theorem}\label{main}
If the Reeb vector field $\cR$ on a regular contact manifold $(M,\zh)$ is complete, then it induces on $M$ either an $\R$- or an $\sU(1)$-principal action, so $p:M\to N=M/\cR$ is a principal bundle for which $\zh$ represents a principal connection whose curvature $\zw$ is a symplectic form on $N$, $p^*(\zw)=\xd\zh$. If $\zr\in(0,+\infty]$ is the minimal period of the flow $\exp(t\cR)$, then the symplectic form $\zw$ satisfies the integrality condition $[\zw/\zr]\in H^2(N,\Z)$.
\end{theorem}
\no  Note also that if $\cR$ induces a non-free $\sU(1)\simeq S^1$-action, then the quotient manifold $M/\sU(1)$ is generally only an orbifold \cite{Kegel:2021}.

\mn In the second part of the paper, we study the question: what are canonical `contact products' of contact manifolds? Of course, the answer cannot be completely trivial, as the Cartesian product of contact manifolds is even-dimensional, so never contact. The solution comes from the contact reduction of the Cartesian products of symplectizations. It goes back to the observation \cite{Bruce:2017,Grabowski:2013} that there exists a one-to-one correspondence between general contact manifolds  $(M,C)$, where $C$ is a contact distribution, and so-called \emph{symplectic $\Rt$-bundles}, i.e., principal $\Rt$-bundles $\zt:P\to M$ equipped with a 1-homogeneous symplectic form $\zw$. Here, $\Rt$ is the multiplicative group of invertible reals. Using $\Rt$ instead of the additive $\R$ is necessary for including non-coorientable contact structures in the picture. As the Cartesian products of symplectic $\Rt$-bundles are canonically symplectic $\Rt$-bundles with respect to the diagonal action of $\Rt$, we get the required product \emph{via} the mentioned correspondence. Note, however, that this concept is canonical only in the distributional setting of contact geometry, as contact products of cooriented contact manifolds are only coorientable, there is no canonical trivialization.

\mn We consider as well another concept of a `contact product', namely, we study the question: what are canonical products of prequantizations? On the level of contactifications of symplectic manifolds, the problem can be translated to seeking canonical products of contactifications. We present a solution to this problem under an additional requirement that if both are finite, the minimal periods $\zr_1,\zr_2$ of the corresponding Reeb vector fields are commensurate, $\zr_2/\zr_1=k/l$, $k,l\in\N$. If $k,l$ are relatively prime, the minimal period $\zr$ of the Reeb vector field of the product is $\zr=\zr_2/k=\zr_1/l$.

\mn The paper is organized as follows. In Section 2 we describe symplectizations of general (possibly non-trivializable) contact structures understood as \emph{symplectic $\Rt$-bundles}. Section 3 is devoted to introducing \emph{contactifications} of symplectic manifolds and related reductions: symplectic-to-contact and contact-to-symplectic. Regular contact manifolds with all Reeb orbits compact are studied in Section 4; we prove a generalization of the Boothby-Wang theorem. In Section 5, in turn, we study the general case of complete Reeb vector fields and prove our main result. Products of contact manifolds and contact relations are defined in Section 6, while in Section 7 we introduce products of principal contact manifolds and prequantization spaces. We end up with a section containing conclusions and an outlook.

\section{Contact structures: symplectizations}
\subsection{Basics}
Generally, a \emph{contact structure} on a manifold $M$ of dimension $(2n+1)$ is a \emph{maximally nonintegrable} distribution $C\subset \sT M$, being a \emph{field of hyperplanes} on $M$, i.e., a distribution with rank $2n$. Such a distribution is, at least locally, the kernel of a nonvanishing 1-form $\zh$ on $M$, i.e., $C=\ker(\zh)$, and the maximal nonintegrability means that the 2-form $\xd\zh$ is nondegenerate on $C$. Such a (local) 1-form we call a \emph{contact form}. It is alternatively characterized by the condition that $\zh\we(\xd\zh)^n$ is nonvanishing, i.e., it is a volume form. Of course, the 1-form $\zh$ is determined only up to conformal equivalence. i.e., a 1-form $\zh$ is a contact form if and only if $f\zh$ is a contact 1-form for any nonvanishing function $f$. The contact form $f\zh$ we call \emph{equivalent} to $\zh$. It defines the same contact distribution $C=\ker(\zh)$. The local picture for contact forms is fully described in the following.
\begin{theorem}[Contact Darboux Theorem] Let $\zh$ be 1-form on a manifold $M$ of dimension $(2n+1)$. Then $\zh$ is a contact form if and only if around every point of $M$ there are local coordinates $(z,p_i,q^i)$, $i=1,\dots,n$, in which $\zh$ reads
\be\label{Dc} \zh=\xd z-p_i\,\xd q^i.\ee
\end{theorem}
\no In this paper, we will consider mainly \emph{trivial} (\emph{cooriented}) contact structures, i.e., manifolds equipped with a globally defined contact form. A contact manifold $(M,C)$ is \emph{trivializable}, i.e., it admits a global contact form $\zh$ with $C=\ker(\zh)$ if and only if the line bundle $L=\sT M/C$ is trivializable.

\no Any contact form $\zh$ on $M$ determines uniquely a vector field $\cR$ on $M$, called the \emph{Reeb vector field}, which is uniquely characterized by the equations
$$i_\cR\zh=1\quad \text{and}\quad i_\cR\xd\zh=0.$$
The Reeb vector field for the contact form (\ref{Dc}) is $\cR=\pa_z$. The Reeb vector field is nonvanishing and respects the contact form, i.e., $\Ll_\cR\zh=0$, where $\Ll$ denotes the Lie derivative.

\subsection{Symplectizations: symplectic $\Rt$-bundles}
For a nonvanishing 1-form $\zh$ on a manifold $M$, we have the following well-known fundamental observation.
\begin{proposition} The 1-form $\zh$ is contact if and only if the 2-form $\zw$ on $\tM=M\ti\R_+$,
\be\label{sy}\zw(x,s)=\xd(s\cdot\zh)(x,s)=\xd s\we\zh(x)+s\cdot\xd\zh(x),
\ee
is a symplectic form. Here, $\R_+$ denotes the multiplicative group of positive reals, and $s\in\R_+$ is the standard coordinates inherited from $\R$.
\end{proposition}
Note that the manifold $\tM$ can be viewed as a trivial principal $\R_+$-bundle $\zt:\tM\to M$, with the $\R_+$-action $h_\zn(x,s)=(x,\zn s)$, and the 2-form $\zw$ is \emph{homogeneous} (more precisely, \emph{$1$-homogeneous}), i.e., $h_\zn^*(\zw)=\zn\cdot\zw$. Moreover, there is a one-to-one correspondence between contact forms equivalent to $\zh$ and other trivializations of the principal bundle, given by (\ref{sy}). If we change the trivialization by a factor $f(x)$, then the vertical coordinate changes by $s\mapsto s/f(x)$ and the contact form by $\zh\mapsto f\zh$, so the 1-form $\zvy=s\zh$ does not depend on the trivialization. Also, the generator $\n=s\pa_s$ does not depend on the trivialization. The 1-form $\zvy$ on $\tM$ we call the \emph{Liouville 1-form}, and the vector field $\n$ the \emph{Liouville vector field}. The principal bundle $\tM$, equipped with the homogeneous symplectic form $\zw$ we call the \emph{symplectization} of the contact form $\zh$. Also the contact distribution $C=\ker(\zh)$ does not depend on the trivialization and can be written as the projection of the kernel of $\zvy$, $C=\sT\zt\big(\ker(\zvy)\big)$. On the other hand, the Reeb vector field $\cR$ strongly depends on the trivialization, as the Reeb vector field $\cR'$ of the equivalent contact form $\zh'=f\zh$ is not $\cR/f$ but $\cR'=\cR/f+X$, where $X\in C$ is uniquely determined by the equation
$$i_X\xd\zh=\xd f-\cR(f)\zh.$$
In this paper, we consider $\R_+$-bundles rather than $\R$-bundles. Of course, both pictures are equivalent, but it is more convenient to see $(\R_+,\cdot)\simeq(\R,+)$ as a subgroup in $\Rt$. The standard symplectization is often considered on $\tM=M\ti\R$ rather than $\tM=M\ti\R_+$, so (\ref{sy}) takes the form
\be\label{hsf1}\zw_\zh(x,t)=\xd(e^t\zh)(x,t)=e^t\big(\xd t\we\zh(x)+\xd\zh(x)\big).\ee

\mn The picture presented above is a particular case of a \emph{symplectic $\Rt$-bundle} in the terminology of \cite{Grabowski:2013} (here, $\Rt=\R\setminus\{0\}$ is the multiplicative group of nonzero reals). Such bundles are defined as $\Rt$-principal bundles $\zt:P\to M$, equipped additionally with a 1-homogeneous symplectic form $\zw$. They are symplectizations of general contact manifolds $(M,C)$. For any local trivialization of $P$, the symplectic form $\zw$ reads exactly like in (\ref{sy}) for a local contact form $\zh$ determining $C$, $C=\ker(\zh)$, associated with the local trivialization. The Liouville 1-form is well defined on $P$, as well as the Liouville vector field $\n$, and $C=\sT\zt\Big(\ker(\zvy)\big)$. There is a one-to-one correspondence between isomorphism classes of symplectic $\Rt$-bundles and contact manifolds. A canonical realization is the $\Rt$-bundle $P=(C^o)^\ti\subset\sT^*M$, where $(C^o)^\ti$ is the submanifold of $\sT^*M$ consisting of nonzero vectors of the annihilator of $C$. The homogeneous symplectic form is in this case the restriction of the canonical symplectic form $\zw_M$ on $\sT^*M$ to $P$ (cf. \cite{Arnold:1989}).
The necessity of using the non-connected group $\Rt$ instead of just $\R_+$ comes from the need to include non-cooriented contact structures into the picture. Closer studies of such structures, together with a description of the corresponding contact Hamiltonian mechanics, can be found in a series of papers \cite{Bruce:2017,Grabowska:2022,Grabowska:2023,Grabowska:2024,Grabowska:2024b,Grabowski:2013}.

\section{Contactifications}
It is obvious that any cooriented contact manifold $(M,\zh)$ of dimension $(2n+1)$ is automatically presymplectic with the exact presymplectic form $\xd\zh$ of rank $2n$. In this case the involutive distribution $\ker(\xd\zh)$ is generated by the Reeb vector field $\cR$. In the following, we will use the terminology of \cite{Boothby:1958}.
\begin{definition}
A cooriented contact manifold $(M,\zh)$ we call \emph{regular} if the foliation $\cF$ of $M$ into $\cR$-orbits is simple, i.e., the space $M/\cR$, of orbits of $\cR$ carries a smooth manifold structure such that the canonical projection $p:M\to M/\cR$ is a surjective submersion. In other words, $p$ is a smooth fibration. We call $(M,\zh)$ \emph{complete} if the Reeb vector field is complete, i.e., its flow $\exp(t\cR)$ is a 1-parameter group of (global) diffeomorphisms of $M$. If $\cR$ is the generator of a principal action of a Lie group $\cG=\sU(1)$ or $\cG=\R$ on $M$, we call the contact manifold $(M,\zh)$ \emph{principal}.
\end{definition}
\no Of course, any principal contact manifold is regular and complete. We will prove the converse of this statement.

\begin{remark} Any regular compact contact manifold is complete. The dynamics of the Reeb vector fields on compact contact manifolds is a subject of intensive studies, partly because of its relation to Hamiltonian dynamics on a fixed energy hypersurface. For a very general geometric approach to contact Hamiltonian mechanics as a part of the classical symplectic Hamiltonian mechanics, we refer to \cite{Grabowska:2022}.
A long-standing open problem concerning the Reeb dynamics is the so-called \emph{Weinstein Conjecture}, stating that for contact forms on compact manifolds, the corresponding Reeb vector fields carry at least one periodic orbit. Note that in \cite{Weinstein:1979} it was supposed additionally that the manifold is simply connected because the author suspected the existence of a counterexample for a torus. The counterexample appeared to be false and this assumption has been finally dropped. This conjecture has been proved for some particular cases, especially for 3-dimensional manifolds \cite{Taubes:2007}. The origins and the history of the Weinstein Conjecture are nicely described in \cite{Pasquotto:2012}.
\end{remark}
\no The contact manifolds considered by Weinstein were hypersurfaces $M$ in a symplectic manifold $(P,\zW)$, equipped with a contact form $\zh$ such that $\zW\,\big|_M=\xd\zh$, where $\zW\,\big|_M$ is the restriction of $\zW$ to $M$. Weinstein called them \emph{hypersurfaces of contact type}. Since any hypersurface in a symplectic manifold is automatically coisotropic, the corresponding Reeb vector field on a hypersurface of contact type spans its characteristic distribution. It is proved in \cite{Weinstein:1979} that any hypersurface of contact type can be obtained by a \emph{symplectic-to-contact reduction}. Contact forms on $M$ associated with the symplectic $\R_+$-bundle $\zt:M\ti\R_+\to M$ are exactly of this type: we identify $M$ with the image of a global section which is automatically a hypersurface of contact type.
\begin{proposition}[symplectic-to-contact reduction] Let $(P,\zW)$ be a symplectic manifold and $M$ be a hypersurface in $P$. If $\n$ is a vector field, defined in a neighbourhood of $M$, such that
$\n$ is transversal to $M$ and $\Ll_\n\zW=\zW$, then the restriction $\zh$ to $M$ of the 1-form
$\tilde\zh=i_\n\zW$ is a contact form on $M$, and $\xd\zh=\zW\,\big|_M$.
\end{proposition}
\begin{proof}
The 2-form $\xd\zh$ is the restriction to $M$ of
$$\xd\tilde\zh=\xd\, i_\n\zW=\Ll_\n\zW=\zW.$$
If $X\in(\ker(\xd\zh)\cap\ker(\zh))$, then $\zW(\n,X)=0$ and $\zW(\sT M,X)=0$, thus $X=0$.
\end{proof}
\no There are various generalizations of the above proposition, see for instance \cite{Grabowski:2004}. We have also a canonical reduction going in the reverse direction.
\begin{proposition}[contact-to-symplectic reduction]
Let $(M,\zh)$ be a regular contact manifold, and let $p:M\to N=M/\cR$ be the corresponding fibration. Then there is a unique symplectic form $\zw$ on $N$ such that $p^*(\zw)=\xd\zh$,
\end{proposition}
\begin{proof}
The kernel of the closed 2-form $\xd\zh$ on $M$ is spanned by $\cR$, so one can apply the standard symplectic reduction of presymplectic manifolds.
\end{proof}
\begin{definition} The procedure of passing from $(M,\zh)$ to $(N,\zw)$ we call a \emph{contact-to-symplectic reduction}, and the contact manifold $(M,\zh)$ -- a \emph{contactification} of the symplectic manifold $(N,\zw)$.
\end{definition}
\no The following example is well known in the literature (see e.g. \cite[Appendix 4]{Arnold:1989}).
\begin{example}\label{ex2}
Let $(N,\zw)$ be an exact symplectic manifold, $\zw=\xd\zvy$.  Then
$$\zh(x,t)=\zvy(x)+\xd t$$
is a contact form on $M=N\ti\R$ and $(M,\zh)$ is a contactification of $(N,\zw)$. In the sequel we will use the multiplicative version of this contactification, replacing $N\ti\R$ with $N\ti\R_+$ and writing
\be\label{exact}
\zh(x,s)=\zvy(x)+\xd s/s.
\ee
In the first version, $\zh$ is invariant as a 1-form on the trivial $\R$-bundle $N\ti\R$, while in the second as living on the trivial $\R_+$-bundle $N\ti\R_+$.
\end{example}

\mn Finding contactifications of compact symplectic manifolds (they are never exact) is generally a more sophisticated task. Note also that contactifications are never unique, since any open submanifold $\cU\subset M$ of a contactification $(M,\zh)$ of $(N,\zw)$ which projects onto the whole $N$ is also a contactification of $(N,\zw)$ with the contact form $\zh\,\big|_\cU$. Particularly interesting are complete contactifications, e.g., compact contactifications of compact symplectic manifolds which cannot be obtained \emph{via} the above procedure.
Standard examples are odd-dimensional spheres.
\begin{example}\label{exsphere}
In quantum physics, the complex projective space $\C\P^{n-1}$ is viewed as the space of pure states in the Hilbert space $\C^n$. The standard basis $(e_k)$ of $\C^n$ induces global real coordinates
$(q^k,p_l)$, $k,l=1,\dots,n$, on $\C^n\simeq\R^{2n}$ by
$$\la e_k\,|\,x\ran=(q^k+i\cdot p_k)(x).$$
In these coordinates, the vector $\pa_{q^k}$ represents $e_k$ and $\pa_{p_k}$ represents
$i\cdot e_k$. The 1-form
\be\label{zvy}\zvy=\sum_k\big(q^k\xd p_k-p_k\xd q^k\big)\ee
we will call the \emph{Liouville 1-form}. Its restriction to the unit sphere $S^{2n-1}\subset\C^n$ is a contact form $\zh$ on $S^{2n-1}$ which turns the sphere into a regular contact manifold with periodic Reeb flow, acting as the diagonal subgroup $\sU(1)$ of $\SU(n)$ on $S^{2n-1}$. The base of this principal bundle is exactly $\C\P^{n-1}$ with its canonical Fubini-Study symplectic form. This reduction in stages is
$$\C\P^{n-1}=(\C^n)^\ti/\C^\ti=(\C^n)^\ti/\big(\R_+\ti\sU(1)\big)=\big((\C^n)^\ti/\R_+\big)/\sU(1)
=S^{2n-1}/\sU(1).$$
\end{example}

\section{Regular contact manifolds with compact orbits}
Let us consider now a regular contact connected manifold $(M,\zh)$, so that $p:M\to N=M/\cR$ is a smooth fibration. The fibers of this fibration consist of orbits of the Reeb vector field $\cR$ (being closed submanifolds in the regular case) which are diffeomorphic either to circles or to $\R$. On every compact orbit $\cO_x$, the flow generated by $\cR$ is periodic with the minimal period $\zr(x)$. Of course, if $M$ is compact, then all orbits, being closed, are circles automatically.

\mn Suppose for a moment that all $\cR$ orbits are circles. In this situation, Ehresmann's fibration theorem \cite{Ehresmann:1951}, stating that smooth fibrations $p:M\to N$ are locally trivial if only $p$ is a proper map (e.g., $M$ is compact), implies that our fibration by compact $\cR$-orbits is actually a locally trivial fibration. Indeed, this follows from the fact that every fiber has a tubular neighbourhood which is relatively compact.

\mn That the flow of $\cR$ is periodic as a whole is the subject of the following proposition, which is just an extension of Boothby-Wang results \cite{Boothby:1958} to the non-compact case. Having periodicity for open contact manifolds is crucial for the next steps of our characterization.
\begin{proposition}\label{pmain} Suppose that, for a connected regular contact manifold $(M,\zh)$, the fibration $p:M\to N=M/\cR$ is actually a fiber bundle over $N$ with the typical fiber $S^1$. Then the flow generated by $\cR$ on $M$ is periodic with the minimal period $\zr$ common for all orbits, which defines a principal action of the group $S^1$ turning $p:M\to N$ into an $S^1$-principal bundle. The contact form $\zh$ can be seen as the connection 1-form of a principal connection, whose curvature $\zw$ is a symplectic form on $N$, and $[\zw/\zr]\in H^2(N,\Z)$ represents the Euler class of this $S^1$-bundle.
\end{proposition}
\begin{proof}
Let us consider a local trivialization $M_U=p^{-1}(U)\simeq U\ti S^1$ of the fiber bundle $p:M\to N$, where $U$ is a connected open subset in $N$ with coordinates $(x^a)\in\R^{2n}$. It will be convenient to consider the standard covering of the circle
$$\R\ni \zt\mapsto [\zt]\in \R/\Z$$
and the corresponding covering
$$\zz:\tU=U\ti\R\to U\ti S^1.$$
It allows us to use coordinates $(x^a,\zt)$ on $\tU$ and to consider functions on $U\ti S^1$ as functions on $\tU$ which are 1-periodic with respect to $\zt$. The pullback of $\zh$ to $\tU$ is a contact form $\th$ with the pullback $\tR$ of $\cR$ as the Reeb vector field.
Let us write $\th$ in coordinates as
$$\th=g(x,\zt)\,\xd \zt+f_a(x,\zt)\,\xd x^a,$$
where $g(x,\zt)$ and $f_a(x,\zt)$ are 1-periodic in $\zt$. Since $\tR$ is tangent to the orbits,  $\tR=\pa_\zt/g$ and
\be\label{int}
\int_0^1g(x,\zt)\,\xd \zt=\zr(x)
\ee
is the minimal period of $\cR$ on $\cO_x$. We also have
$$
\xd\th=\frac{\pa g}{\pa x^a}(x,\zt)\,\xd x^a\we\xd \zt+\frac{\pa f_a}{\pa \zt}(x,\zt)\,\xd \zt\we\xd x^a
+\frac{\pa f_a}{\pa x^b}(x,\zt)\,\xd x^b\we\xd x^a.
$$
Since $i_\tR\xd\th=0$,
$$\Big(\frac{\pa g}{\pa x^a}(x,\zt)-\frac{\pa f_a}{\pa \zt}(x,\zt)\Big)\,\xd x^a=0,$$
so
$$\frac{\pa g}{\pa x^a}(x,\zt)=\frac{\pa f_a}{\pa \zt}(x,\zt)$$
for all $a$. Consequently (cf. (\ref{int})),
\beas &&\frac{\pa \zr}{\pa x^a}(x)=\frac{\pa}{\pa x^a}\Big(\int_0^1g(x,\zt)\,\xd \zt\Big) =\int_0^1\frac{\pa g}{\pa x^a}(x,\zt)\,\xd \zt\\
&&=\int_0^1\frac{\pa f_a}{\pa \zt}(x,\zt)\,\xd \zt=f_a(x,1)-f_a(x,0)=0,
\eeas
because $f_a(x,\zt)$ are 1-periodic with respect to $\zt$. Hence, $\zr(x)$ is constant on $U$. But $M$ is connected, so $\zr(x)$ is globally constant, $\zr(x)=\zr$. We have $\exp(t\cR)=\id$
if and only if $t\in\Z_\zr$, where $\Z_\zr=\zr\cdot\Z$. In other words, $p:M\to N$ is a $\T_\hbar$-principal bundle in the terminology of \cite{Bates:1997}, where $\hbar=\zr/2\pi$.

\m Now, we will change the coordinates in $\tU$ into $(x^a,t)$, parametrizing fibers of $\tU$ by trajectories of the lifted Reeb vector field $\tR=\pa_\zt/g$,
$$\big(x,t(x,\zt)\big)=\Big(x,\int_0^\zt g(x,s)\xd s\Big).$$
A direct inspection shows that the diffeomorphism $(x,\zt)\mapsto(x,t(x,\zt))$ maps the vector field $\tR$ onto $\pa_t$.  Hence, in the coordinates $(x^a,t)$ our contact form reads
\be\label{zhU}\th=g(x,\zt)\xd\zt+f_a(x,\zt)\xd x^a=\xd t+h_a(x,t)\,\xd x^a\ee
for some functions $h_a$. By a direct check we get that $h_a(x,t)=f_a(x,0)$, so
$h_a(x,t)=h_a(x)$ does not depend on $t$ for all $a$, and we get
\be\label{zhU1}\xd\th=\frac{\pa h_a}{\pa x^b}(x)\,\xd x^b\we\xd x^a.\ee
Finally, since $\xd\zh$ depends only on coordinates $(x^a)$, it is the pullback of a uniquely determined 2-form $\zw_U$ on $U$ having in coordinates $(x^a)$ formally the form (\ref{zhU1}). As $\xd\th$ is of rank $2n$, the form $\zw_U$ is symplectic. From the uniqueness of $\zw_U$ it follows that there is a globally defined symplectic form $\zw$ on $N$ such that $\zw\big|_U=\zw_U$, and $(M,\zh)$ is a contactification of $\zw$. Because for $\zvy_U=h_a(x)\,\xd x^a$ we have $\zw_U=\xd\zvy_U$,
each $\zw_U$ is exact but, clearly, $\zw$ need not to be exact globally.

\mn To prove the integrality condition $\big[\zw/\zr\big]\in H^2(N,\Z)$ (cf. also \cite{Kobayashi:1956} and \cite[Section 7.2]{Geiges:2008}), let us choose a \v Cech cover $\{U_\za\}$ of $N$ (all intersections of the cover members are connected and contractible), so the $S^1$-bundles $p:M_\za=p^{-1}(U_\za)\to U_\za$ are trivial $S^1$-principal bundles, and equip $\tU_\za$ with local coordinates $(x^a_\za,t_\za)$ as above. The contact form $\zh$ in these coordinates reads (cf. (\ref{zhU}))
$$\zh_\za=\widetilde{\zh}_{U_\za}=\xd t_\za+\zvy_\za(x).$$
On the intersection $U_{\za\zb}=U_\za\cap U_\zb$ consider coordinates $(x^a,t_\za)$ and $(x^a,t_\zb)$, respectively, where $(x^a)$ are local coordinates on $U_{\za\zb}$, the same for $U_\za$ and $U_\zb$. Since the diffeomorphism $(x^a,t_\za)\mapsto(x^a,t_\zb)$ corresponds to an isomorphism of $S^1$-principal bundles, we have $t_\zb(x,t_\za)=t_\za+A_{\zb\za}(x)$ for some function $A_{\zb\za}:U_{\za\zb}\to\R$. Of course, $A_{\za\zb}=-A_{\zb\za}$. Since the shift of $t_\za$ by
$$T_{\zg\zb\za}=A_{\za\zg}+A_{\zg\zb}+A_{\zb\za}$$ induces the identity on $M_\za$, we have the cocycle condition
\be\label{cc}T_{\zg\zb\za}\in\zr\cdot\Z.\ee
Since,
$$\zvy_\zb-\zvy_\za=\xd A_{\zb\za},$$
in view of the canonical isomorphism between the de Rham and the \v Cech cohomology,  (\ref{cc}) implies that the cohomology class of the closed 2-form $\zw/\zr$ in $H^2(N;\R)$ lies in $H^2(N;\Z)$,
\be\label{ic}\big[\zw/\zr\big]\in H^2(N;\Z).\ee
Such closed 2-forms $\zw$ are called \emph{$\zr$-integral}, and it is known that they are characterized by the property $\int_\zS\zw\in\Z_\zr$ for each closed 2-dimensional surface $\zS$ in $N$.

\end{proof}
\begin{remark}\label{prequant}
In the geometric quantization (see, for instance, \cite{Bates:1997,Brylinski:2008,Sniatycki:1980,Woodhouse:1992}), as $\zr$ is usually taken $2\pi\hbar$, where $\hbar$ is the Planck constant, and $[\zw/2\pi\hbar]\in H^2(N,\Z)$ is the well-known Weil integrality condition for the existence of a prequantum bundle on the symplectic manifold $(N,\zw)$. This is because there is an obvious one-to-one correspondence between $S^1$-principal bundles on $N$ and complex Hermitian line bundles $\C\hookrightarrow \L\to N$ associated with the standard $S^1=\sU(1)$-linear and isometric action on $\C$. The contact manifold $M$ is embedded in $\L$ as the submanifold of vectors with length 1. Prequantum states are sections $\psi$ of $\L$, and classical observables (Hamiltonians) $H:N\to \R$ give rise to linear operators on $\Sec(\L)$ given by
\be\label{QO}
\wh{H}(\psi)=-i\hbar\,\sD_{X_H}\psi+H\psi,
\ee
where $X_H$ is the Hamiltonian vector field of $H$.

It is well known that sections $\psi$ of $\L$ can be identified with $\sU(1)$-equivariant complex functions $F:M\to\C$ (for a real analog of this construction with applications to contact mechanics see \cite[Section 2]{Grabowska:2024b}),
$$F\big(\exp(t\cR)(y)\big)=e^{-2\pi it/\zr}\cdot F(y).$$
In our case, we have additionally a principal connection represented by the connection 1-form $\zh$, which induces a linear connection $\sD$ on $\L$ by
$$\sD_XF=X^h(F),$$
where $X^h$ is the horizontal lift of the vector field $X$ on $N$. This definition is correct, i.e., $X^h(F)$ is again equivariant, as the Reeb vector field $\cR$ commutes with the horizontal lifts of vector fields on $N$.
The Hermitian product of equivariant functions $F$ and $G$ is then represented by the function $F\ol{G}$ which is, clearly, $\sU(1)$-invariant, so it represents a function on $N$. The prequantum Hilbert space is the completion of the space of compactly supported sections with the scalar product
$$\la F,G\ran=\int_NF\ol{G}\cdot\zw^n,$$
where $\zw^n$ is the volume form associated with the symplectic form $\zw$.

Note that $[\zw/2\pi\hbar]\in H^2(N,\Z)$ represents the first Chern class of $\L$, which completely characterizes $\L$ up to isomorphism and does not depend on the choice of the connection and the metric \cite[Chapter 2]{Brylinski:2008}. Note also that the fact that $\zr=2\pi\hbar$ is the minimal period of $\cR$ is reflected in the notation in \cite{Bates:1997}, where the discrete additive subgroup $2\pi\hbar\cdot\Z$ of $\R$ is denoted $\Z_\hbar$ and the corresponding circle group $\T_\hbar=\R/\Z_\hbar$. In this sense, the $\sU(1)$ action induced by the flow of $\cR$ is understood as an action of $\T_\hbar$, i.e., a $2\pi\hbar$-periodic action of $\R$, and the integrality condition takes the form $[\zw]\in H^2(N,\Z_\hbar)$.
\end{remark}
\section{The general case}
Let us assume again that $(M,\zh)$ is a regular contact manifold and $\cR$ is a complete vector field, thus its flow is global and generates a smooth action $(t,x)\mapsto\exp(t\cR)(x)$ on $M$ of the group $(\R,+)$ of additive reals. \emph{A priori}, the dynamics of $\cR$ could contain both, compact and non-compact orbits. We will show that this is not possible if $\cR$ is complete. Note, however, that without the completeness assumption, such examples do exist. For instance, one can take a regular contact compact manifold (like $S^{2n-1}$ in Example \ref{exsphere}) and remove, say, one point from a fiber. Of course, after removing this point the Reeb vector field is no longer complete. To finish the proof of Theorem \ref{main}, we will first consider the case when all orbits of $\cR$ are non-compact.

\subsection{No compact orbits}
Suppose first that all orbits are non-compact. In this case, the $\R$-action induced by $\cR$ is free. Moreover, $p:M\to N=M/\cR$ is a fibration with fibers homeomorphic to $\R$, thus automatically a fiber bundle (see \cite[Corollary 31]{Meigniez:2002}). It is easy to see that the free $\R$-action of the flow of $\cR$ which respects the fibers of this fiber bundle is automatically proper, so it turns this fiber bundle into an $\R$-principal bundle. Locally, $p^{-1}(U)=U\ti\R$, and using the flow of $\cR$ to parametrize the fibers, we get local coordinates $(x^a,t)$ on $U\ti\R$ such that $\cR=\pa_t$. In other words, the $\R$-action on $U\ti\R$ is $s.(x,t)=(x,t+s)$. This form of $\cR$ implies that $\zh$ can be locally written as
$$\zh=\xd t+f_a(x,t)\,\xd x^a.$$
Hence,
$$\xd\zh=\frac{\pa f_a}{\pa t}(x,t)\,\xd t\we\xd x^a
+\frac{\pa f_a}{\pa x^b}(x,t)\,\xd x^b\we\xd x^a,$$
and because $i_\cR\xd\zh=0$, we get $\frac{\pa f_a}{\pa t}(x,t)=0$ for all $a$. It follows that the functions $f_a$ do not depend on $t$, $f_a(x,t)=f_a(x)$, so $\zh=\xd t+\zvy_U$, where $\zvy_U=f_a(x)\xd x^a$ is the pullback of a 1-form on $U$. Since
\be\label{symp}\xd\zh=\xd\zvy_U=\frac{\pa f_a}{\pa x^b}(x)\,\xd x^b\we\xd x^a\ee
is of rank $2n$, it is the pullback of a uniquely determined symplectic form $\zw_U$ on $U$ which
in coordinates $(x^a)$ looks exactly like (\ref{symp}). Being uniquely determined by $\xd\zh$, the symplectic forms $\zw_U$ agree on the intersections $U_1\cap U_2$, so that there is a symplectic form $\zw$ on $N$ such that $\xd\zh=p^*(\zw)$. In other words, $(M,\zh)$ is a contactification of $(N,\zw)$ on which the $\R$-action generated by the flow of $\cR$ defines an $\R$-principal bundle structure. Moreover, $\zw$ represents the curvature of the principal connection $\zh$.

The fiber bundle $p:M\to N$ is trivializable, as the fibers are contractible. Using a global section $\zs:N\to M$ to identify $N$ with a submanifold $\zs(N)$ in $M$, we can view $M\simeq N\ti\R$ as a trivial $\R$-principal bundle over $N$. Let $\zh_\zs$ be the contact form $\zh$ reduced to the horizontal submanifold $\zs(N)$. By the identification $N\simeq \zs(N)$ given by the section $\zs$ (or the projection $p$) we can view $\zh_\zs$ as a 1-form on $N$. Since the pullback by $p$ of $\xd\zh_\zs$ is $\xd\zh$, we have $\xd\zh_\zs=\zw$, so the symplectic form $\zw$ is exact. Therefore, it is easy to see that $(M,\zh)$ is the standard contactification of the symplectic form $\zw=\xd\zh_\zs$ described in Example \ref{ex2}.
\begin{remark}\label{prequant1}
The prequantization picture in the above case is much simpler than in the case of $\sU(1)$-action. To work with the multiplicative notation as in the case of $\sU(1)$-bundle, we can consider $M=N\ti\R_+$ and $\zh(x,s)=\zvy(x)+\xd s/s$. In this case $\zr=+\infty$ and $[\zw/\zr]=0\in H^2(N,\hbar\,\Z)$ for any $\hbar>0$. The prequantum bundle is the trivial line bundle, $\L=N\ti\C$, associated with the principal $\R_+$-bundle $p:M\to N$ with respect to the $\R_+$-action on $\C$ by multiplication. The linear connection $\sD$ on $\L$ associated with the trivial connection on $N$ is the trivial connection $\sD_X\psi=X(\psi)$, so we can put
$$\wh{H}_\hbar(\psi)=-i\hbar\,X_H(\psi)+H\psi,$$
which works now for any $\hbar>0$.
\end{remark}
\subsection{There exists a compact orbit}
For a result being a variant of the celebrated \emph{Reeb local stability theorem} and describing the behavior of smooth fibrations near a compact fiber, we refer to Meigniez \cite[Lemma 22]{Meigniez:2002}. It simply says that, for a smooth fibration, every compact subset of every fibre has a product neighborhood. In our situation, this immediately implies that if $x$ is a point in $N$ for which the orbit $\cO_x=p^{-1}(x)$ is compact, then there is a (connected) neighbourhood $U\subset N$ of $x$ such that $p$ is a fiber bundle when restricted to $p^{-1}(U)$. In other words, any compact orbit has a neighbourhood $M_U=p^{-1}(U)$ in which $p$ is a trivializable fiber bundle, $M_U\simeq U\ti S^1$, with the typical fiber $S^1$.

\mn It follows now from Proposition \ref{pmain} that in the open submanifold $M_U\subset M$ the Reeb vector field induces an $S^1$-principal action. Let us fix such $U$ and denote the corresponding period $\zr$. Let $$M_{\zr}=\{x\in N:\,\zr(x)=\zr\}$$
be the set of points of $N$ for which $\cO_x$ is a compact orbit with the minimal period $\zr$.
It is clear from what has been said that $M_{\zr}$ is open. We will show that it is also closed.

\mn Indeed, let $x_0\in N$ belong to the closure of $M_{\zr}$ and $y_0\in\cO_{x_0}$. The submersion $p:M\to N$ is an open map, so in a neighbourhood of $y_0$ there is a sequence of points $(y_n)$ such that $y_n\rightarrow y_0$ and $p(y_n)=x_n\in M_{\zr}$. Since $\cR$ is complete, its flow $\zf_t=\exp(t\cR)$ is globally defined for all $t\in\R$. We have then
$$y_n=\zf_\zr(y_n)\rightarrow\zf_\zr(y_0).$$
Hence, $\zf_\zr(y_0)=y_0$, so $\cO_{x_0}$ is also a periodic orbit with period $\zr$.
This is, in fact, the minimal period for $\cO_{x_0}$, since the minimal period (Proposition \ref{pmain}) is locally constant on periodic orbits.
For connected $N$, all this implies that if there is one periodic orbit of $\cR$ with the minimal period $\zr$, then all orbits are periodic with the same minimal period $\zr$. This finishes the proof of Theorem \ref{main}.

\section{Products of contact manifolds and contact relations}
The concept of a \emph{contact product} of contact manifolds is essentially known, although it was not extensively studied in the literature. Recently \cite{Grabowska:2024a}, it was also defined for Sasaki manifolds. Of course, the contact product cannot live on the Cartesian product of the manifolds, as the latter is of even dimension. For cooriented manifolds, a concept of such a product appeared in \cite{Ibanez:1997}, however, the contact form for this product is chosen \emph{ad hoc}; in fact, there are infinitely many possibilities for such a choice. This is because this product is defined canonically only in the distributional setting of contact geometry. Indeed, there exists a canonical approach related to the obvious product of symplectic $\Rt$-bundles, which are symplectizations of contact manifolds \cite{Arnold:1989,Bruce:2017,Grabowski:2013}.

Let us start with products of principal $G$-bundles. The product bundle structure
$$\zt=(\zt_1,\zt_2):P_1\ti P_2\to M_1\ti M_2,$$
where $\zt_i:P_i\to M_i$ is a principal $G_i$-bundle with a $G_i$-action $h^i:G_i\ti M_i\to M_i$, $i=1,2$, is
canonically a principal $G_1\ti G_2$-bundle, with the principal action
$$\bh_{(g_1,g_2)}(y_1,y_2)=\big(h^1_{g_1}(y_1),h^2_{g_2}(y_2)\big).$$
Actually, this is the simplest example of a \emph{double principal bundle} \cite{Grabowska:2018,Lang:2021}. Note that the `compatibility' of principal actions of groups $G_1$ and $G_2$ in the definition of double principal bundles is not just a commutation of these actions, as every group action should be compatible with itself. Canonical non-product examples are the frame bundles of double vector bundles \cite{Grabowska:2018}; for a very effective approach to double vector bundles see \cite{Grabowski:2009,Grabowski:2012}.

A particularly interesting situation is when $G_1=G_2=G$. In this case, we can consider the diagonal action
$$\hat h_g(y_1,y_2)=\big(h^1_g(y_1),h^2_{g}(y_2)\big),$$
which makes the product manifold $P_1\ti P_2$ into a principal $G$-bundle, denoted $P_1\bti P_2$,
$$\bar\zt:P_1\bti P_2\to M_1\bt M_2=(P_1\ti P_2)/G,$$
where the quotient is with respect to the diagonal action.
Moreover, we have obvious canonical $G$-bundle morphisms of $P_1\bti P_2$ onto $P_1$ and $P_2$. This construction produces canonically a $G$-bundle out of two $G$-bundles and it is easy to see that it is a product in the category of principal $G$-bundles. Note that, in the case of $G=\Rt$, the category of $\Rt$-bundles is equivalent to the category of line bundles, and $\bti$-products of $\Rt$-bundles correspond to similar products in the category of line bundles  (see \cite{Le:2018,Schnitzer:2023,Schnitzer:2023a,Zapata:2020}).

\medskip Suppose now that $G$ is commutative. Let us observe that in this case, the manifold $M_1\bt M_2$ is  itself a principal bundle with the canonical action of the quotient group $\hat G=(G\ti G)/G$, with $G\subset G\ti G$ being the diagonal normal subgroup. If $[(y_1,y_2)]$ is the class of $(y_1,y_2)\in P_1\ti P_2$, and $[(g_1,g_2)]$ is the class of $(g_1,g_2)\in G_1\ti G_2$, then
$$\sh_{[(g_1,g_2)]}\big([(y_1,y_2)]\big)=\big[\big(h^1_{g_1}(y_1),h^2_{g_2}(y_2)\big)\big].$$
It is easy to check that this definition makes sense.
The group $\hat G$ is isomorphic with $G$, however, in many ways; no one is manifestly privileged. For instance, such isomorphisms can be obtained from the two canonical embeddings $G\to G\ti G$ into the first and the second factor, respectively, but these isomorphisms give two opposite actions $\bar h^1_g=\bar h^2_{g^{-1}}$ of $G$ on $P_1\ti P_2$,
$$\bar h^1_g\big([(y_1,y_2)]\big)=\big[\big(h^1_{g}(y_1),y_2\big)\big]\quad\text{and}\quad
\bar h^2_g\big([(y_1,y_2)]\big)=\big[\big(y_1, h_g^2(y_2)]\big)\big].
$$
Consequently, $M_1\boxtimes M_2$ is a principal $\hat G$-bundle over the base manifold $(P_1\ti P_2)/(G\ti G)=M_1\ti M_2$, with the projection
$$[\zt]:M_1\boxtimes M_2\to M_1\ti M_2,$$
but there is no privileged $G$-action. We can use, for instance, those coming from $\bar h^1$ or $\bar h^2$.

\mn Let now $G=\Rt$, let $(M_i,C_i)$, $i=1,2$, be contact manifolds, and let $\zt_i:P_i\to M_i$ be corresponding symplectic $\Rt$-bundles (symplectizations), with $\Rt$-actions $h^1,h^2$, and  homogeneous symplectic forms $\zw_1,\zw_2$. It is clear that the symplectic form
$$(\zw_1\op\zw_2)(y_1,y_2)=\zw_1(y_1)+\zw_2(y_2)$$
on $P_1\ti P_2$ is homogeneous with respect to the diagonal action of $\Rt$, so, according to our philosophy, it is a symplectization of a contact manifold $(M,C)$, where $M=M_1\bt M_2$ and
\be\label{ncs}C=C_1\boxtimes C_2=\sT\bar\zt\big(\ker(\zvy)\big).
\ee
Here, $\zvy$ is the Liouville 1-form on the product symplectic $\Rt$-bundle $(P_1\bti P_2,\zw_1\op\zw_2)$, being the sum of the Liouville 1-forms on $P_1$ and $P_2$,
$$\zvy(y_1,y_2)=(\zvy_1\op\zvy_2)(y_1,y_2)=\zvy_1(y_1)+\zvy_2(y_2),$$
and the kernel of $\zvy$ is invariant with respect to the diagonal action, thus projects onto the contact distribution $C_1\boxtimes C_2$ on $M_1\boxtimes M_2$.
This way, we obtained a canonical procedure of constructing a contact manifold $(M_1\bt M_2,C_1\bt C_2)$ out of two contact manifolds. The resulting contact manifold we will call the \emph{contact product of contact manifolds} $(M_i,C_i)$, $i=1,2$.
\begin{remark}
Note that the notation $M_1\boxtimes M_2$ (resp., $C_1\boxtimes C_2$) is a little bit misleading, as this product is not determined exclusively by $M_1$ and $M_2$. In fact, both objects give rise to a well-defined contact manifold $(M_1,C_1)\bt(M_2,C_2)$.
\end{remark}
\begin{example}\label{prodcont} To see the above construction in local coordinates, let us take the product of local trivializations of $P_1$ and $P_2$, and consider the associated bundle coordinates $(x_1,x_2,s_1,s_2)$. In these coordinates,
$$\zw(x_1,x_2,s_1,s_2)=\xd s_1\we\zh_1(x_1)+s_1\cdot\xd\zh_1(x_1)
+\xd s_2\we\zh_2(x_2)+s_2\cdot\xd\zh_2(x_2)$$
and
$$\zvy(x_1,x_2,s_1,s_2)=s_1\cdot\zh_1(x_1)+s_2\cdot\zh_2(x_2),$$
where $\zh_1$ and $\zh_2$ are the contact forms on $M_1$ and $M_2$, respectively, associated with the trivializations.
After the reduction by the diagonal action of $\Rt$ on $P_1\ti P_2$, out of the action $\bar h^1$ we get a trivialization of the principal $\Rt$-bundle $[\zt]:M_1\bt M_2\to M_1\ti M_2$ with local coordinates $(x_1,x_2,t)$, where $t\in\Rt$. This corresponds to the identification $(\Rt\ti\Rt)/\Rt\simeq\Rt$ associated with the parametrization $\Rt\ti\Rt\ni(t,s)\mapsto(st,s)$.
The contact form associated with this parametrization is
\be\label{pcf}
\zh(x_1,x_2,t)=t\zh_1(x_1)+\zh_2(x_2).
\ee
The Reeb vector field is $\cR_2$ and the contact distribution is the Whitney sum
\be\label{c1}C_1\bt C_2=C_1\op C_2\op\la \cR_1-t\cR_2,\pa_t\ran.\ee
If we use the parametrization of $M_1\bt M_2$ associated with the $\Rt$-action $\bar h^2$ instead of $\bar h^1$, we get
$$\zh'(x_1,x_2,t')=\zh_1(x_1)+t'\zh_2(x_2).
$$
The Reeb vector field is now $\cR_1$ and
\be\label{c2}C_1\bt C_2=C_1\op C_2\op\la t'\cR_1-\cR_2,\pa_{t'}\ran.\ee
Of course, the contact distribution $C_1\bt C_2$ is determined uniquely, and (\ref{c1}) is the same as (\ref{c2}), since $t'=1/t$ and
$$t'\cR_1-\cR_2=\frac{1}{t}\big(\cR_1-t\cR_2),\quad \text{and}\quad\pa_{t'}=-t^2\pa_t,$$
so the distributions $\la \cR_1-t\cR_2,\pa_t\ran$ and $\la t'\cR_1-\cR_2,\pa_{t'}\ran$ are the same.
The contact forms $\zh$ and $\zh'$ are different but equivalent, since they define the same contact distribution.
This confirms the advantage of working with contact distributions and not contact forms, since we are not tied to a particular trivialization. Moreover, the contact distribution $C_1\bt C_2$ is canonical, while the corresponding contact forms like $\zh$ or $\zh'$ can be arbitrary.
\end{example}

\begin{remark}
Let us remark that the above procedure can be applied \emph{mutatis mutandis} to general Jacobi structures \cite{Vitagliano:2019}, which are understood as \emph{Poisson $\Rt$-bundles} \cite{Bruce:2017,Marle:1991,Vitagliano:2018}. In \cite{Ibanez:1997}, the products have been defined for trivial bundles, i.e., Jacobi brackets in the sense of Lichnerowicz \cite{Lichnerowicz:1978}, and trivial contact manifolds as a particular case, however, the attachment to the established contact form and trivialization makes the picture a little bit less geometrical.

\mn Note also that the two obvious projections
$$(M_1,C_1)\leftarrow (M_1\bt M_2,C_1\bt C_2)\rightarrow (M_2,C_2)$$
provide a canonical example of a \emph{contact dual pair}, a concept studied in \cite{Blaga:2020}, and being closely related to concepts of Jacobi and Poisson dual pairs and the Morita equivalence \cite{Weinstein:1983,Xu:1991}.
We will not go deeper into the corresponding theory here, postponing it to a separate paper.
\end{remark}
\no Having defined the products, we can propose a natural definition of a contact relation.
\begin{definition}
A \emph{contact relation} between contact manifolds $(M_i,C_i)$, $i=1,2$, is a Legendrian submanifold $\cL$ in the contact product $(M_1,C_1)\bt(M_2,C_2)$.
\end{definition}
\begin{example} Consider a contactomorphism between contact manifolds $(M_1,C_1)$ and $(M_2,C_2)$, i.e.,
a diffeomorphism $\zf:M_1\to M_2$ respecting the contact distributions,
\be\label{ccm}\sT_y\zf\big(C_1(y)\big)=C_2\big(\zf(y)\big).
\ee
Locally (or in a cooriented case), if $\zh_1$ and $\zh_2$ are contact forms representing the contact distributions, this means that
\be\label{cm}\zf^*(\zh_2)=f\zh_1,
\ee
for a nonvanishing function $f$ on $M_1$. If $(P_i,\zw_i)$ is a symplectic cover of $(M_i,C_i)$, $i=1,2$, then $\zf$ is covered by an isomorphism $\zF:P_1\to P_2$ of symplectic $\Rt$-bundles (cf. \cite{Grabowska:2023}). We can as well take $\zF$ to be an anti-symplectomorphism, since $(P_2,\zw_2)\simeq(P_2,-\zw_2)$. The graph of $\zF$, $\Graph(\zF)$, is then a Lagrangian submanifold in the product $(P_1\ti P_2,\zw_1\ominus\zw_2)$ of symplectic manifolds. This graph is $\Rt$-invariant, so its projection onto the contact manifold $(M_1\bt M_2,C_1\bt C_2)$ is a Legendre submanifold $\Graph_c(\zf)$, the contact relation associated with $\zf$.
For local trivializations in Example \ref{prodcont}, we get
$$\Graph_c(\zf)=\big\{(x_1,x_2,s)\in M_1\ti M_2\ti\Rt\,\big|\, x_2=\zf(x_1),\ s=-f\big\}.$$
It is easy to see that the contact form (\ref{pcf}) vanishes on $\Graph_c(\zf)$, so the latter is a Legendre submanifold.
\end{example}
\section{Products of prequantization bundles}
In the previous section, we defined a contact product of contact manifolds, whose dimension is by 1 greater than the sum of dimensions of these manifolds. Here, we will define a product of principal contactifications $(M_i,\zh_i)$ of symplectic manifolds $(N_i,\zw_i)$, $i=1,2$, which is a contact manifold whose dimension is by 1 smaller than the sum of dimensions of $M_1$ and $M_2$.

\mn First, let us assume that the minimal periods of the Reeb vector fields $\cR_1$ and $\cR_2$ are equal and finite, say, $\zr\in\R_+$. Consider the Cartesian product manifold $M_1\ti M_2$ endowed with the product 1-form $\big(\zh_1\op\zh_2\big)$.
Of course, $(\zh_1\op\zh_2)$ living on an even-dimensional manifold cannot be contact. Its kernel is
$$\ker(\zh_1)\op\ker(\zh_2)\op\la \cR_1-\cR_2\ran,$$
so its characteristic distribution (cf. \cite[Definition 2.5]{Grabowska:2024}) is
$$\zq(\zh_1\op\zh_2)=\ker(\zh_1\op\zh_2)\cap\ker\big(\xd(\zh_1\op\zh_2)\big)=\la \cR_1-\cR_2\ran.$$
Since $\cR_1$ and $\cR_2$ are commuting, linearly independent, and periodic with the same minimal period $\zr$, the vector field $\cR_1-\cR_2$ is again periodic with the minimal period $\zr$. Moreover, since the trajectories of $\cR_1$ and $\cR_2$ form regular foliations being fibers of $S^1$-principal bundles, the flow of the vector field $\cR_1-\cR_2$ induces also a principal $S^1$-action. In other words, the 1-form $(\zh_1\op\zh_2)$ is \emph{simple} in the terminology of \cite[Definition 2.5]{Grabowska:2024}), so it admits a regular reduction \cite[Theorem 2.6]{Grabowska:2024}) to a contact form $\zh$ on the reduced manifold $M=(M_1\ti M_2)/\la\cR_1-\cR_2\ran$. The reduced contact manifold $(M,\zh)$ we will call the \emph{product of principal contact manifolds} $(M_i,\zh_i)$, $i=1,2$, and denote $\big(M_1\uti M_2,\zh_1\uti\zh_2\big)$. If, in local coordinates like (\ref{zhU}), we have
$$\zh_1(x,t)=\xd t+h_a(x)\xd x^a,\quad\text{and}\quad \zh_2(y,s)=\xd s+g_i(y)\xd y^i,$$
then
$$\zh(x,y,t)=\xd t+h_a(x)\xd x^a+g_i(y)\xd y^i.$$
Vector fields $X$ on $M_1\ti M_2$, being affine combinations of $\cR_1$ and $\cR_2$, $X=a\cR_1+b\cR_2$ with $a+b=1$, commute with $\cR_1-\cR_2$, thus project onto the Reeb vector field $\cR$ of $\zh$. The Reeb vector field is again periodic with the minimal period $\zr$ and induces a principal $S^1$ action on $M$.
This way, we have obtained a new regular contact manifold, being a contactification of the symplectic manifold
$(N_1\ti N_2,\zw_1\op\zw_2)$, which we will call the \emph{product of principal contactifications}.

\begin{example}
We know from Example \ref{exsphere} that odd-dimensional contact spheres $S^{2n+1}$ are $\sU(1)$-principal contactifications of complex projective spaces $\C\P^n$ with the minimal period $2\pi$. Hence, $M=(S^{2n+1}\ti S^{2m+1})/\sU(1)$, where the $\sU(1)$-action is diagonal, carries a canonical structure of a $\sU(1)$-principal contactification of the Cartesian product $\C\P^n\ti\C\P^m$ of complex projective spaces.
\end{example}
\no Note that our construction of $M_1\uti M_2$ corresponds to the tensor product on the level of prequantizations. To see this, let us consider the complex line bundle $\zp_j:\L_j\to N_j$ corresponding to the $\sU(1)$ principal bundle $\zt_j:M_j\to N_j$, $j=1,2$. Recall that sections of $\L_j$ are identified with complex functions $F_j:M_j\to\C$ which are equivariant,
$$F_j\big(\exp(t\cR_j)(y_j)\big)=e^{-2\pi it/\zr}\cdot F_j(y_j).$$
Sections of the tensor product line bundle
\be\label{ppqq}\zp:\L_1\ot\L_2\to N_1\ti N_2,\ee
are generated by tensor products of sections of $\L_1$ and sections of $\L_2$, which in our identifications,
are products
$$(F_1\ot F_2)(y_1,y_2)=F_1(y_1)\cdot F_2(y_2)$$
of equivariant functions in $M_1$ and $M_2$, respectively. They are formally defined on $M_1\ti M_2$, but because $F_1$ and $F_2$ are equivariant, they are constant on the paths
$$\big(\exp(t\cR_1)(y_1),\exp(-t\cR_2)(y_2)\big),\quad t\in\R,$$
so they effectively live on $M_1\uti M_2$, and are equivariant with the induced $\sU(1)$-action there,
$$\Big(\exp\Big(\frac{t}{2}\cR_1\Big)(y_1),\exp\Big(\frac{t}{2}\cR_2\Big)(y_2)\Big)
=e^{-2\pi it/\zr}\cdot F_1(y_1)\cdot F_2(y_2).$$
It is easy to see that the Hermitian metric is the product metric as well as the connection. Summing up these observations we get the following.
\begin{theorem}
For any principal contactifications $(M_j,\zh_j)$ of symplectic manifolds $(N_j,\zw_j)$, $i=1,2$, with the same minimal period $\zr$ of the corresponding Reeb vector fields $\cR_j$, $i=1,2$, its product
$$M_1\uti M_2=(M_1\ti M_2)/(\cR_1-\cR_2)$$ endowed with the reduced contact form
$$\zh_1\uti\zh_2=(\zh_1\op\zh_2)/(\cR_1-\cR_2)$$
is a principal contactification of the product symplectic manifold $(N_1\ti N_2,\zw_1\op\zw_2)$. The corresponding
prequantization bundle is the tensor product complex line bundle (\ref{ppqq}) with the product Hermitian metric
$$\La\psi_1(x_1)\ot\psi_2(x_2),\psi'_1(x_1)\ot\psi'_2(x_2)\Ra
=\La\psi(x_1),\psi(x_2)\Ra\cdot\La\psi'(x_1),\psi'(x_2)\Ra$$
and the product linear connection
$$\sD_{X_1\op X_2}\psi_1\ot\psi_2=\sD^1_{X_1}\psi_1\ot\psi_2+\psi_1\ot\sD^2_{X_2}\psi_2.$$
\end{theorem}
\no Of course, if $(M,\zh)$ is a principal contact manifold with the minimal period $\zr\in(0,+\infty]$ and $c\in\Rt$, then $(M,c\zh)$ is also a principal contact manifold with the period $c\zr$ for any $c\in\R_+$. Of course, this does not change the Chern class, $\zw/\zr=c\zw/c\zr$.
Hence, we can consider the product of $(M_1,\zh_1)$ and $(M_2,c\zh_2)$ such that $\frac{\zr_1}{c\zr_2}=1$.

\mn There are some obvious generalizations of the above procedure of constructing the product of principal contact manifolds. First, we do not need to assume that the minimal periods $\zr_1$ and $\zr_2$ of the Reeb flows are the same. Of course, due to the integrality conditions,
$$[\zw_1\op\zw_2]\in H^2(N_1\ti N_2,\zr_1\cdot\Z+\zr_2\cdot\Z)$$
need not be integral in general, as $\zr_1\cdot\Z+\zr_2\cdot\Z$ can be a dense subgroup of $\R$, and we are forced to assume that they are commensurate.

The product of fibers of the principal bundles $\zt_j:M_j\to N_j$, $i=1,2$ is the two-dimensional torus $\T^2=S^1\ti S^1$, and the vector field of $\cR_1-\cR_2$ on $\T^2=\R^2/\Z^2$ is represented on $\R^2$, with coordinates $(x,y)$, by the vector field $X=\za\pa_x-\zb\pa_y$, where $\za=2\pi/\zr_1$ and $\zb=2\pi/\zr_2$.
It is well known that the trajectories of $X$ projected onto $\T^2$ form a simple foliation if and only if $\za/\zb$ is a rational number, say $k/l$, where $k,l\in\Z$ are relatively prime, $(k,l)=1$. In this case, we can perform the contact reduction as above getting a contact manifold $(M,\zh)$. The vector field $a\cR_1+b\cR_2$ is represented on $\R^2$ by $Y=a\za\pa_x+b\zb\pa_y$ and projects onto the Reeb vector field  $\cR$ of the reduced contact form $\zh$ if and only if $a+b=1$.
\begin{lemma} The Reeb vector field $\cR$ is periodic with the minimal period
$$\zr=\frac{\zr_2}{k}=\frac{\zr_1}{l}.$$
\end{lemma}
\begin{proof} We are looking for the minimal $t>0$ such that the curve $\zg(t)=\big(a\za t,b\zb t\big)$ in $\R^2$ intersects the set
$$A=\{(s\za,-s\zb)+\Z^2\,\big|\, s\in\R\}\subset\R^2.$$
This set consists of parallel lines, and the one closest to the origin among those not going through the origin is the line $y=x_0-\frac{l}{k}x$ with the smallest possible $x_0>0$, which exists due to the periodicity of $X$. There are integers $m,n$ and $s\in\Rt$, such that
$$(x, x_0-lx/k)=(m,n),$$ so $x_0-lm/k=n$. In other words, we are looking for the smallest $x_0>0$ such that
$kx_0=kn+lm$ for some integers $m,n$. But $k,l$ are relatively prime, and by the \emph{B\'ezout's identity} the minimal positive value of $kn+lm$ with integer $m,n$ is $1$, so the minimal $x_0$ is $x_0=1/k$. Now, we just have to find $t>0$ such that $(a\za t,b\zb t)$ is the point on the line $y=1/k-lx/k$. We get $t=1/(k\zb)=\zr_2/k$, which is the minimal period of $\cR$.

\end{proof}
\no If, in turn, one of the principal contact structures is as in Example \ref{ex2}, $M_1=N_1\ti\R$ with the contact form $\zh_1(y,t)=\zvy(y)+\xd t$, then the situation is even simpler. The Reeb vector field $\cR_1$ on $M_1$ is $\pa_t$ with the period $+\infty$, and the vector field $X=\cR_1-\cR_2$ on $M_1\ti M_2$ defines automatically a principal $\R$-action, so $M_1\uti M_2=N_1\ti M_2$ and $\zh=\zvy\op\zh_2$, with the Reeb vector field $\cR=\cR_2$, whose period is $\zr_2$.
Summing up our observations, we get the following.
\begin{theorem}
Let $(M_j,\zh_j)$ be a principal contactification of a symplectic manifold $(N_j,\zw_j)$, with the Reeb vector field $\cR_j$ with the minimal period $\zr_j\in(0,+\infty]$, $i=1,2$. Suppose that $\zr_1=+\infty$ or both $\zr_j$ are finite with
$$\frac{\zr_2}{\zr_1}=\frac{k}{l},$$
where $k,l$ are positive, relatively prime integers. Then, the action of flow $\exp\big(t(\cR_1-\cR_2)\big)$ on $M_1\ti M_2$ is free and proper, the 1-form $\zh_1\op\zh_2$ is $(\cR_1-\cR_2)$-invariant, and the quotient manifold $M_1\uti M_2$ equipped with the reduced 1-form $\zh_1\uti\zh_2$
is a principal contactification of the symplectic manifold $(N_1\ti N_2,\zw_1\op\zw_2)$ with the minimal period $\zr=+\infty$ if $\zr_1=\infty$, and
$$\zr=\frac{\zr_2}{k}=\frac{\zr_1}{l}$$
in the other case. The corresponding Reeb vector field is the projection of the vector field $(\cR_1+\cR_2)/2$ on $M_1\uti M_2$.
\end{theorem}
\begin{remark}
Let us note that there are some similar constructions of products proposed for quasi-regular Sasakian manifolds in \cite{Boyer:2007} (see also \cite{Wang:1990}).
\end{remark}

%%%%%%%%%%%%%%%%%%%%%%%%%%%%%%%%
\section{Conclusions and outlook}
The idea of this paper came from our interests in contact supergeometry \cite{Bruce:2017,Grabowski:2013}, contact mechanics \cite{Grabowska:2022,Grabowska:2023}, and geometry of quantum states \cite{Grabowski:2005}. We wanted to have an extension of the celebrated Boothby-Wang theorem \cite{Boothby:1958} to the non-compact case, which is difficult to find in the literature in an explicit form; even corrected versions of the Boothby-Wang proof in \cite{Geiges:2008,Niederkruger:2005} work exclusively for compact manifolds. The compactness assumption appears to be irrelevant, and we provided a very simple proof. Moreover, the fact that regular contact manifolds with compact orbits of the Reeb vector field are canonically principal bundles extends to the case of the complete Reeb vector field, which is satisfied automatically on compact manifolds. Although regular contact manifolds can contain simultaneously compact and non-compact orbits of the Reeb vector field, we proved it is impossible in the case of complete Reeb vector fields. Consequently, regular contact manifolds $(M,\zh)$ with complete Reeb vector fields are automatically principal bundles with the structure group either $\sU(1)$ or $\R$ in which $\zh$ defines a principal connection. In both cases, the curvature $\zw$ of this connection is a symplectic form on the base manifold $N=M/\cR$ and satisfies an integrality condition. Due to a canonical one-to-one correspondence between principal $\sU(1)$-bundles and complex Hermitian line bundles over $N$, we get a direct connection to the geometric quantization of symplectic manifolds and quantum physics.

Being interested in products of quantizations, we studied products in the category of contact manifolds and introduced the concept of a product also for regular contact manifolds with complete vector fields, which can be translated to an idea of products of prequantum bundles. This raises interesting questions concerning the geometric quantization for composite systems, which should include the corresponding polarizations into the picture and need much more complex geometry. We postpone these problems to a separate paper.

%%%%%%%%%%%%%%%%%%%%%%%%%%%%%%%%%%%%
\section{Acknowledgements}
The authors thank Alan Weinstein for his clarifications concerning the Weinstein Conjecture.

\vskip.5cm
\noindent Katarzyna Grabowska\\\emph{Faculty of Physics,
University of Warsaw,}\\
{\small ul. Pasteura 5, 02-093 Warszawa, Poland} \\{\tt konieczn@fuw.edu.pl}\\
https://orcid.org/0000-0003-2805-1849\\

\noindent Janusz Grabowski\\\emph{Institute of Mathematics, Polish Academy of Sciences}\\{\small ul. \'Sniadeckich 8, 00-656 Warszawa,
Poland}\\{\tt jagrab@impan.pl}\\  https://orcid.org/0000-0001-8715-2370


\begin{thebibliography}{V}

\bibitem{Arnold:1989} V.~I.~Arnold,
\newblock{Mathematical Methods of Classical Mechanics,}
\newblock{\emph{Graduate Texts in Mathematics} \textbf{60}, Springer-Verlag, New York, 1989.}

\bibitem{Bates:1997} S.~Bates, A.~Weinstein,
\newblock{Lectures on the geometry of quantization,}
\newblock{\emph{Berkeley Mathematics Lecture Notes} \textbf{8}. \textsl{American Mathematical Society, Providence, RI; Berkeley Center for Pure and Applied Mathematics, Berkeley, CA,} 1997.}

\bibitem{Blaga:2020} A.~M.~Blaga, M.~A.~Salazar, A.~G.~Tortorella, C.~Vizman,
\newblock{Contact dual pairs,}
\newblock{\emph{Int. Math. Res. Not.} IMRN(2020), no. 22, 8818--8877.}

\bibitem{Boothby:1958} W.~M.~Boothby, H.~C.~Wang,
\newblock{On contact manifolds,}
\newblock{\emph{Ann. of Math.} \textbf{68} (1958), 721--734.}

\bibitem{Boyer:2008} Ch.~P.~Boyer, K.~Galicki,
\newblock{Sasakian Geometry,}
\newblock{Oxford Mathematical Monographs. Oxford University Press, Oxford (2008).}

\bibitem{Boyer:2007} Ch.~P.~Boyer, K.~Galicki, L.~Ornea,
\newblock{Constructions in Sasakian geometry,}
\newblock{\emph{Math. Z.} \textbf{257} (2007), 907--924.}

%\bibitem{Bravetti:2021} H.~Cruz-Prado, A.~Bravetti, A.~Garcia-Chung,
%\newblock{From geometry to coherent dissipative dynamics in quantum mechanics,}
%\newblock{\emph{Quantum Rep.} \textbf{3} (2021), 664--683.}

\bibitem{Bruce:2017} A.~J.~Bruce, K.~Grabowska, J.~Grabowski,
\newblock{Remarks on contact and Jacobi geometry,}
\newblock{\emph{SIGMA Symmetry Integrability Geom. Methods Appl.} \textbf{13} (2017), Paper No. 059, 22 pp.}

\bibitem{Brylinski:2008} J.-L.~Brylinski,
\newblock{Loop spaces, characteristic classes and geometric quantization,}
\newblock{\emph{Mod. Birkh\"auser Class.} Birkh\"auser Boston, Inc., Boston, MA, 2008.}

\bibitem{Ehresmann:1951} C.~Ehresmann,
\newblock{Les connexions infinit\'esimales dans un espace fibr\'e diff\'erentiable (French),} \newblock{\textsl{Colloque de topologie (espaces fibr\'es), Bruxelles, 1950}, pp. 29--55. Georges Thone, Li\`ege; Masson \& Cie, Paris, 1951.}

\bibitem{Geiges:2008} H.~Geiges,
\newblock{An introduction to contact topology,}
\newblock{\emph{Cambridge Stud. Adv. Math.} \textbf{109},
Cambridge University Press, Cambridge, 2008.}

\bibitem{Grabowska:2018} K.~Grabowska, J.~Grabowski,
\newblock{$n$-tuple principal bundles,}
\newblock{\emph{Int. J. Geom. Methods Mod. Phys.} \textbf{15} (2018), 1850211, 18 pp.}

\bibitem{Grabowska:2022} K.~Grabowska, J.~Grabowski,
\newblock{A geometric approach to contact Hamiltonians and contact Hamilton-Jacobi Theory.}
\newblock{\emph{J. Phys. A} \textbf{55} (2022), 435204 (34pp).}

\bibitem{Grabowska:2023} K.~Grabowska, J.~Grabowski,
\newblock{Reductions: precontact versus presymplectic,}
\newblock{\emph{Ann. Mat. Pura Appl.} \textbf{202} (2023), 2803--2839.}

\bibitem{Grabowska:2024} K.~Grabowska, J.~Grabowski, M.~Ku\'s, G.~Marmo,
\newblock{Contactifications: a Lagrangian description of compact Hamiltonian systems,}
\newblock{\emph{J. Phys. A} \textbf{57} (2024), 395204 (31pp).}

\bibitem{Grabowska:2024b} K.~Grabowska, J.~Grabowski,
\newblock{Contact geometric mechanics: the Tulczyjew triple,}
\newblock{\emph{Adv. Theor. Math. Phys.} \textbf{28} (2024), 599--654.}

\bibitem{Grabowska:2024a} K.~Grabowska, J.~Grabowski, R.~Mohseni,
\newblock{Sasaki structures on general contact manifolds,}
\newblock{\emph{arXiv:2412.16697}.}

\bibitem{Grabowski:2004} J.~Grabowski, D.~Iglesias, J.~C.~Marrero, E.~Padr\'on, P.~Urba\'nski,
\newblock{Poisson-Jacobi reduction of homogeneous tensors,}
\newblock{\emph{J. Phys. A} \textbf{37} (2004), 5383--5399.}

\bibitem{Grabowski:2013} J.~Grabowski,
\newblock{Graded contact manifolds and contact Courant algebroids}
\newblock{\emph{J. Geom. Phys. } \textbf{68} (2013), 27--58.}

\bibitem{Grabowski:2005} J.~Grabowski, M.~Ku\'s, G.~Marmo,
\newblock{Geometry of quantum systems: density states and entanglement,}
\newblock{\emph{J. Phys. A} \textbf{38} (2005), 10217--10244.}

\bibitem{Grabowski:2009} J.~Grabowski, M.~Rotkiewicz,
\newblock{Higher vector bundles and multi-graded symplectic manifolds,}
\newblock{\emph{J. Geom. Phys.} \textbf{59} (2009), no. 9, 1285--1305}

\bibitem{Grabowski:2012} J.~Grabowski, M.~Rotkiewicz,
\newblock{Graded bundles and homogeneity structures,}
\newblock{\emph{J. Geom. Phys.} \textbf{62} (2012), no. 1, 21--36.}

%\bibitem{Hermann:1960} R.~Hermann,
%\newblock{A sufficient condition that a mapping of Riemannian manifolds be a fiber bundle,}
%\newblock{\emph{Proc. Amer. Math. Soc.} \textbf{11} (1960), 236--242.}

%\bibitem{Hofer:1993} H.~Hofer,
%\newblock{Pseudoholomorphic curves in symplectizations with applications to the Weinstein conjecture in dimension three,}
%\newblock{\emph{Invent. Math.} \textbf{114} (1993), 515--563.}

\bibitem{Ibanez:1997} R.~Ib\'a\~nez, M.~de León, J.~C.~Marrero, D.~Mart\'{\i}n de Diego,
\newblock{Co-isotropic and Legendre-Lagrangian submanifolds and conformal Jacobi morphisms,}
\newblock{\emph{J. Phys. A} \textbf{30} (1997), 5427--5444.}

\bibitem{Kegel:2021} M.~Kegel, Ch.~Lange,
\newblock{A Boothby-Wang theorem for Besse contact manifolds,}
\newblock{\emph{Arnold Math. J.} \textbf{7} (2021), 225--241.}

\bibitem{Kobayashi:1956} S.~Kobayashi,
\newblock{Principal fibre bundles with the 1–dimensional toroidal group,}
\newblock{\emph{T\^ohoku Math. J.} (2) \textbf{8} (1956), 29--45.}

\bibitem{Lang:2021} H.~Lang, Y.~Li, Z.~Liu,
\newblock{Double principal bundles,}
\newblock{\emph{J. Geom. Phys.} \textbf{170} (2021), Paper No. 104354, 19 pp.}

\bibitem{Le:2018} H\^ong V\^an L\^e, Yong-Geun Oh, A.~G.~Tortorella, L.~Vitagliano,
\newblock{Deformations of coisotropic submanifolds in Jacobi manifolds,}
\newblock{\emph{J. Symplectic Geom.} \textbf{16} (2018), 1051--1116.}

\bibitem{Lichnerowicz:1978} A.~Lichnerowicz,
\newblock{Les vari\'et\'es de Jacobi  et leurs alg\'ebres de Lie associ\'ees,}
\newblock{\emph{J. Math. Pures Appl.}, \textbf{57} (1978), 453--488.}

\bibitem{Marle:1991} C.~M.~Marle,
\newblock{On Jacobi manifolds and Jacobi bundles,}
\newblock{in Symplectic geometry, groupoids, and integrable systems (Berkeley, CA, 1989), 227--246, \emph{Math. Sci. Res. Inst. Publ.} \textbf{20}, Springer, New York, 1991.}

%\bibitem{Marsden:1982} J.~E.~Marsden, A.~Weinstein, T.~Ratiu, R.~Schmid, R.~G.~Spencer,
%\newblock{Hamiltonian systems with symmetry, coadjoint orbits and plasma physics,}
%\newblock{Proceedings of the IUTAM-ISIMM symposium on modern developments in analytical mechanics, Vol. I (Torino, 1982), \emph{Atti Accad. Sci. Torino Cl. Sci. Fis. Mat. Natur.} \textbf{117} (1983), suppl. 1, 289--340.}

\bibitem{Meigniez:2002} G.~Meigniez,
\newblock{Submersions, fibrations and bundles,}
\newblock{\emph{Trans. Amer. Math. Soc.} \textbf{354} (2002), 3771--3787.}
%Lemma 22, Corollary 31

%\bibitem{Meyer:1973} K.~R.~Meyer,
%\newblock{Symmetries and integrals in mechanics,}
%\newblock{\emph{Dynamical systems} (Proc. Sympos., Univ. Bahia, Salvador, 1971), pp. 259--272. Academic Press, New York, 1973.}

\bibitem{Niederkruger:2005} K.~Niederkr\"uger,
\newblock{Compact Lie group actions on contact manifolds,}
\newblock{Inaugural-Dissertation, Universit\"at zu K\"oln (2005).}

\bibitem{Pasquotto:2012} F.~Pasquotto,
\newblock{A short history of the Weinstein conjecture,}
\newblock{\emph{Jahresber. Dtsch. Math.-Ver.} \textbf{114} (2012), 119--130.}

\bibitem{Schnitzer:2023} J. Schnitzer,
\newblock{Normal forms for Dirac-Jacobi bundles and splitting theorems for Jacobi structures,}
\newblock{\emph{Math. Z.} \textbf{303} (2023), Paper No. 74, 31 pp.}

\bibitem{Schnitzer:2023a} J.~Schnitzer, A.~G.~Tortorella,
\newblock{Weak dual pairs in Dirac-Jacobi geometry,}
\newblock{\emph{Commun. Contemp. Math.} \textbf{25} (2023), Paper No. 2250035, 56 pp.}

\bibitem{Sullivan:1978} D.~Sullivan,
\newblock{A foliation of geodesics is characterized by having no “tangent homologies”,}
\newblock{\emph{J. Pure Appl. Algebra} \textbf{13} (1978), 101--104.}

\bibitem{Sniatycki:1980} J.~\'Sniatycki,
\newblock{Geometric quantization and quantum mechanics,}
\newblock{\emph{Appl. Math. Sci.}, \textbf{30}, Springer-Verlag, New York-Berlin, 1980.}

%\bibitem{Reinhart:1959} B.~L.~Reinhart,
%\newblock{Foliated manifolds with bundle-like metrics,}
%\newblock{\emph{Ann. of Math.} \textbf{69} (1959), 119--132.}

\bibitem{Taubes:2007} C.~H.~Taubes,
\newblock{The Seiberg-Witten equations and the Weinstein conjecture,}
\newblock{\emph{Geom. Topol.} \textbf{11} (2007), 2117--2202.}

\bibitem{Wadsley:1975} A.~W.~Wadsley,
\newblock{Geodesic foliations by circles,}
\newblock{\emph{J. Differential Geom.} \textbf{10} (1975), 541--549.}

\bibitem{Vitagliano:2018} L.~Vitagliano,
\newblock{Dirac-Jacobi bundles,}
\newblock{\emph{J. Symplectic Geom.} \textbf{16} (2018), 485--561.}

\bibitem{Vitagliano:2019} L.~Vitagliano,
\newblock{Products in Jacobi Geometry,}
\newblock{conference talk, \url{http://www.dipmat2.unisa.it/people/vitagliano/www/Rio_19.pdf}}

\bibitem{Wang:1990} McKenzie Y.~Wang, W.~Ziller,
\newblock{Einstein metrics on principal torus bundles,}
\newblock{\emph{J. Differential Geom.} \textbf{31} (1990), 215--248.}

\bibitem{Weinstein:1979} A.~Weinstein,
\newblock{On the hypotheses of Rabinowitz' periodic orbit theorems,}
\newblock{\emph{J. Differential Equations} \textbf{33} (1979), 353--358.}

\bibitem{Weinstein:1983} A.~Weinstein,
\newblock{The local structure of Poisson manifolds,}
\newblock{\emph{J. Differential Geom.} \textbf{18} (1983), 523--557.}

\bibitem{Woodhouse:1992} N.~M.~J.~Woodhouse,
\newblock{Geometric quantization,}
\newblock{\emph{Oxford Math. Monogr.}, Oxford Sci. Publ.,
The Clarendon Press, Oxford University Press, New York, 1992.}

\bibitem{Xu:1991} P.~Xu,
\newblock{Morita equivalence of Poisson manifolds,}
\newblock{\emph{Comm. Math. Phys.} \textbf{142} (1991), 493--509.}

\bibitem{Zapata:2020} C.~Zapata-Carratal\'a,
\newblock{Jacobi geometry and Hamiltonian mechanics: the unit-free approach,}
\newblock{\emph{Int. J. Geom. Methods Mod. Phys.} \textbf{17} (2020), 2030005, 87 pp.}

\end{thebibliography}
\end{document}